\documentclass[11pt]{amsart}

\usepackage{hyperref}
\usepackage{fullpage}
\usepackage[table]{xcolor}
\usepackage{tikz}

\def\nrs{\cellcolor{black!25}} % new result

\numberwithin{equation}{section}

\theoremstyle{plain}
\newtheorem{theorem}{Theorem}[section]

\newtheorem{conjecture}[theorem]{Conjecture}
\newtheorem{corollary}[theorem]{Corollary}

\newtheorem{lemma}[theorem]{Lemma}

\newtheorem{proposition}[theorem]{Proposition}

\newtheorem{problem}[theorem]{Problem}

\theoremstyle{definition}
\newtheorem{example}[theorem]{Example}

\newtheorem{definition}[theorem]{Definition}

\def\lmlt#1{\mathrm{LMlt}(#1)}
\def\aut#1{\mathrm{Aut}(#1)}

\def\dis#1{\mathrm{Dis}(#1)}
\def\lmlt#1{\mathrm{Mlt}_{\ell}(#1)}
\def\inn#1{\mathrm{Inn}(#1)}
\def\tab{\phantom{xxxx}}

\begin{document}

\title{Enumeration of racks and quandles up to isomorphism}

\author{Petr Vojt\v{e}chovsk\'y}

\author{Seung Yeop Yang}

\email[Vojt\v{e}chovsk\'y]{petr@math.du.edu}

\email[Yang]{seungyeop.yang@knu.ac.kr}

\address{Department of Mathematics, University of Denver, 2390 S York St, Denver, Colorado, 80208, USA}

\address{Department of Mathematics, Kyungpook National University, Daegu, 41566, Republic of Korea}

\begin{abstract}
Racks and quandles are prominent set-theoretical solutions of the Yang-Baxter equation. We enumerate racks and quandles of orders $n\le 13$ up to isomorphism, improving upon the previously known results for $n\le 8$ and $n\le 9$, respectively. The enumeration is based on the classification of subgroups of small symmetric groups up to conjugation, on a representation of racks and quandles in symmetric groups due to Joyce and Blackburn, and on a number of theoretical and computational observations concerning the representation. We explicitly find representatives of isomorphism types of racks of order $\le 11$ and quandles of order $\le 12$. For the remaining orders we merely count the isomorphism types, relying in part on the enumeration of $2$-reductive racks and $2$-reductive quandles due to Jedli\v{c}ka, Pilitowska, Stanovsk\'y and Zamojska-Dzienio.
\end{abstract}



\thanks{Petr Vojt\v{e}chovsk\'y supported by a 2015 PROF grant of the University of Denver.}

\maketitle

\section{Introduction}

A groupoid $(X,*)$ is a \emph{left quasigroup} if every left translation
\begin{displaymath}
    L_x:X\to X, \quad y\mapsto yL_x = x*y
\end{displaymath}
is a bijection of $X$. A left quasigroup is a \emph{rack} if the left self-distributive law
\begin{displaymath}
    x*(y*z) = (x*y)*(x*z)
\end{displaymath}
holds. A rack is a \emph{quandle} if it is idempotent, i.e., if
\begin{displaymath}
    x*x=x
\end{displaymath}
holds.

Racks and quandles form well-studied classes of set-theoretical solutions of the Yang-Baxter equation \cite{Drinfeld}. Moreover, racks and quandles appear in low-dimensional topology as invariants of oriented knots and links \cite{Joyce, Matveev}. The textbook \cite{EN} offers a friendly introduction to the theory of quandles.

In this paper we enumerate racks and quandles of order $n\le 13$ up to isomorphism, improving upon previously known enumerations for $n\le 8$ for racks and $n\le 9$ for quandles. We also make all isomorphism types of racks of order $n\le 11$ and quandles of order $n\le 12$ available online.

\subsection{Notation}

Let $X$ be a nonempty set and let $S_X$ be the symmetric group on $X$. If a group $G$ acts on $X$ and $x\in X$, we denote by $xG$ the orbit of $x$, by $G_x$ the stabilizer of $x$, and by $X/G$ a complete set of orbit representatives. The set $X/G$ is not uniquely determined but we will assume that one such set has been fixed whenever $X$ and $G$ are given.

For $f$, $g\in S_X$ and $G\le S_X$, we let $g^f = f^{-1}gf$, $G^f = f^{-1}Gf$ and $f^G = \{f^g:g\in G\}$. As usual, let $C_G(H)$ and $N_G(H)$ be the centralizer and the normalizer of $H\subseteq S_X$ in $G$.

For a groupoid $(X,*)$ and $x\in X$, let $L_x$ or $L_x^*$ be the left translation by $x$, the latter notation being used when we need to keep track of the operation. Similarly, $R_x$ or $R_x^*$ will denote the right translation by $x$ in $(X,*)$. The automorphism group of $(X,*)$ will be denoted by $\mathrm{Aut}(X,*)$.

For a left quasigroup $(X,*)$, let
\begin{align*}
    \lmlt{X,*} &= \langle L_x:x\in X\rangle \le S_X,\\
    \dis{X,*} &= \langle L_x^{-1}L_y:x,\,y\in X\rangle\le\lmlt{X,*}
\end{align*}
be the \emph{left multiplication group} and the \emph{displacement group} of $(X,*)$, respectively.\footnote{In rack and quandle literature, the left multiplication group $\lmlt{X,*}$ of a rack $(X,*)$ is often denoted by $\inn{X,*}$ and is called the \emph{inner automorphism group} of $(X,*)$, a terminology that is in conflict with older conventions for quasigroups and loops \cite{Bruck}.} Note that racks can be equivalently defined as groupoids $(X,*)$ satisfying $\lmlt{X,*}\le\aut{X,*}$.

A rack $(X,*)$ is said to be \emph{$2$-reductive} if
\begin{equation}\label{Eq:2red}
    (x*u)*v = (y*u)*v
\end{equation}
holds for every $x$, $y$, $u$, $v\in X$. A rack that is not $2$-reductive will be called \emph{non-$2$-reductive}.

It is not difficult to see that the following conditions are equivalent for a rack $(X,*)$:
\begin{itemize}
\item $(X,*)$ is $2$-reductive, that is, the identity \eqref{Eq:2red} holds,
\item $(X,*)$ satisfies the identity $(x*u)*v = u*v$,
\item $\lmlt{X,*}$ is commutative.
\end{itemize}

A rack $(X,*)$ is \emph{medial} if it satisfies the medial law
\begin{displaymath}
    (x*u)*(v*y) = (x*v)*(u*y).
\end{displaymath}
One can show (see for instance \cite[Proposition 2.4]{HSV}) that a rack $(X,*)$ is medial if and only if $\dis{X,*}$ is commutative. In particular, every $2$-reductive rack is medial.

A medial rack $(X,*)$ is $2$-reductive if and only if it satisfies the identity $(x*y)*y=y*y$. Therefore, a medial quandle $(X,*)$ is $2$-reductive if and only if it satisfies the identity $(x*y)*y=y$. This last identity is used in \cite{JPSZ} as a definition of $2$-reductivity in the context of medial quandles.

\subsection{Asymptotic growth}

The asymptotic growth of racks and quandles is known. Due to the nature of the estimate, it does not matter whether the algebras in question are counted up to isomorphism or absolutely on a fixed set.

Denote by $r(n)$ (resp. $q(n)$) the number of racks (resp. quandles) of order $n$ up to isomorphism. Blackburn \cite{Blackburn} proved that there are constants $c_1=1/4$ and $c_2=(1/6)\log_2(24)+(1/2)\log_2(3)$ such that for every $\varepsilon>0$ and for all sufficiently large orders $n$ we have
\begin{displaymath}
    2^{c_1 n^2- \varepsilon}\le q(n)\le r(n)\le 2^{c_2 n^2+\varepsilon}.
\end{displaymath}
The lower bound is obtained by exhibiting a large class of racks with $\lmlt{X,*}$ isomorphic to an elementary abelian $2$-group, while the upper bound is based on an estimate for the number of subgroups of symmetric groups and on a relatively straightforward analysis of partitions of $n$. In the same paper, Blackburn also developed a representation of racks and quandles in symmetric groups that deserves to be better known and that we recall in Section \ref{Sc:Representation}.

Ashford and Riordan \cite{AshfordRiordan} improved Blackburn's upper bound and showed that for every $\varepsilon>0$ and for all sufficiently large orders $n$ we have
\begin{displaymath}
    2^{n^2/4-\varepsilon} \le q(n)\le r(n)\le 2^{n^2/4+\varepsilon}.
\end{displaymath}
The main idea for their upper bound is harder to convey. Roughly speaking, a small amount of global information partitions the class of all racks defined on $\{1,\dots,n\}$ into relatively small subsets, and the racks in each of those subsets can then be fully determined by an additional small set of parameters.

\subsection{Exact enumeration}

The exact values of $r(n)$ and $q(n)$ are known only for very small values of $n$.

A brute-force approach (constructing one row of the multiplication table at a time and checking whether the resulting partial groupoid is a partial rack) is feasible for $n\le 7$ or so. Henderson, Macedo and Nelson determined $q(n)$ for $n\le 8$ \cite{HMS}. McCarron reported the values $r(n)$ for $n\le 8$ and $q(n)$ for $n\le 9$ in the Online Encyclopedia of Integer Sequences \cite{OEIS}. Elhamdadi, Macquarrie and Restrepo \cite{EMR} also calculated the values $q(n)$ for $n\le 9$ while investigating automorphism groups of quandles.

Jedli\v{c}ka, Pilitowska, Stanovsk\'y and Zamojska-Dzienio \cite{JPSZ} developed a theory of so-called affine meshes in order to construct and enumerate medial and $2$-reductive quandles of small orders. This allowed them to count medial quandles of order $n\le 13$ and $2$-reductive quandles of order $n\le 16$ up to isomorphism.

A quandle is said to be \emph{connected} if its left multiplication group acts transitively on the underlying set. All connected quandles of order less than $36$ (resp. $48$) were obtained by Vendramin \cite{Ven} (resp. in \cite{HSV}). A library of connected quandles of order less than $48$ is available in \texttt{Rig}, a \texttt{GAP} \cite{GAP} package developed by Vendramin.

\subsection{Summary of results}

\begin{small}

\begin{table}[ht]
\begin{displaymath}
\begin{array}{|r|rrrrrrr|}
    \hline
    n                           &1&2&3&4    &5  &6      &7\\
    r(n)                        &1&2&6&19   &74 &353    &2080\\
    r_{\textrm{med}}(n)         &1&2&6&18   &68 &329    &1965\\
    r_{\textrm{$2$-red}}(n)     &1&2&5&17   &65 &323    &1960\\
    r_{\textrm{non-$2$-red}}(n) &0&0&1&2    &9  &30     &120\\
    \hline
    n                           &8      &9          &10             &11             &12                     &13                         &14\\
    r(n)                        &16023  &\nrs159526 &\nrs2093244    &\nrs36265070   &\nrs\mathit{836395102} &\nrs\mathit{25794670618}   &?\\
    r_{\textrm{med}}(n)         &15455  &\nrs155902 &\nrs2064870    &\nrs35982366   &\nrs\mathit{832699635} &\nrs\mathit{25731050872}   &?\\
    r_{\textrm{$2$-red}}(n)     &15421  &155889     &2064688        &35982357       &\mathit{832698007}     &\mathit{25731050861}       &\mathit{1067863092309}\\
    r_{\textrm{non-$2$-red}}(n) &602    &\nrs3637   &\nrs28556      &\nrs282713     &\nrs3697095            &\nrs 63619757              &?\\
    \hline
\end{array}
\end{displaymath}
\caption{The number of racks $r(n)$, medial racks $r_{\textrm{med}}(n)$, $2$-reductive racks $r_{\textrm{$2$-red}}(n)$ and non-$2$-reductive racks $r_{\textrm{non-$2$-red}}(n)$ of order $n$ up to isomorphism.}\label{Tb:Racks}
\end{table}

\end{small}

\begin{small}

\begin{table}[ht]
\begin{displaymath}
\begin{array}{|r|rrrrrrr|}
    \hline
    n                           &1          &2          &3          &4          &5          &6          &7\\
    q(n)                        &1          &1          &3          &7          &22         &73         &298\\
    q_{\textrm{med}}(n)         &1          &1          &3          &6          &18         &58         &251\\
    q_{\textrm{$2$-red}}(n)     &1          &1          &2          &5          &15         &55         &246\\
    q_{\textrm{non-$2$-red}}(n) &0&0&1&2    &7  &18     &52\\
    \hline
    n                           &8          &9          &10             &11             &12             &13                     &14\\
    q(n)                        &1581       &11079      &\nrs102771     &\nrs1275419    &\nrs21101335   &\nrs\mathit{469250886} &?\\
    q_{\textrm{med}}(n)         &1410       &10311      &98577          &1246488        &20837439       &\mathit{466087635}     &?\\
    q_{\textrm{$2$-red}}(n)     &1398       &10301      &98532          &1246479        &20837171       &\mathit{466087624}     &\mathit{13943041873}\\
    q_{\textrm{non-$2$-red}}(n) &183        &778        &\nrs4239       &\nrs28940      &\nrs264164     &\nrs3163262            &?\\
    \hline
\end{array}
\end{displaymath}
\caption{The number of quandles $q(n)$, medial quandles $q_{\textrm{med}}(n)$, $2$-reductive quandles $q_{\textrm{$2$-red}}(n)$ and non-$2$-reductive quandles $q_{\textrm{non-$2$-red}}(n)$ of order $n$ up to isomorphism.}\label{Tb:Quandles}
\end{table}

\end{small}

Our enumerative results are summarized in Tables \ref{Tb:Racks} and \ref{Tb:Quandles}. In Table \ref{Tb:Racks}, $r(n)$ (resp. $r_{\textrm{med}}(n)$, $r_{\textrm{$2$-red}}(n)$ and $r_{\textrm{non-$2$-red}}(n)$) is the number of racks (resp. medial racks, $2$-reductive racks and non-$2$-reductive racks) of order $n$ up to isomorphism. Obviously, $r(n) = r_{\textrm{$2$-red}}(n) + r_{\textrm{non-$2$-red}}(n)$, but we report all three numbers for the convenience of the reader, to better indicate which results are new, and for future reference. The notation in Table \ref{Tb:Quandles} is analogous but for quandles instead of racks.

New results are reported in shaded cells. If a number in the tables is in roman font, representatives of isomorphism types can be downloaded from the website of the first author. If a number in the tables is in italics, representatives of isomorphism types are not available. The numbers that are both in unshaded cells and it italics are either taken from \cite{JPSZ} or they were provided to us by Jedli\v{c}ka in personal communication. For instance, we constructed $3163262$ representatives of isomorphism types of non-$2$-reductive quandles of order $13$, the number $466087624$ of $2$-reductive quandles of order $13$ is taken from \cite{JPSZ}, resulting in the $469250886$ quandles of order $13$ up to isomorphism.

\subsection{Commented outline of the paper}

In Section \ref{Sc:Conjugacy} we show that a classification of left quasigroups up to isomorphism defined on $X$ can be accomplished by independent classifications of left quasigroups with a given left multiplication group $G\le S_X$. Crucially, it suffices to consider subgroups $G$ of $S_X$ up to conjugation in $S_X$, rather than all subgroups of $S_X$.

\begin{table}[ht]
\begin{displaymath}
    \begin{array}{|c|rrrrrrrrrrrrr|}
        \hline
        n&1&2&3&4&5&6&7&8&9&10&11&12&13\\
        a(n)&1&2&4&11&19&56&96&296&554&1593&3094&10723&20832\\
        b(n)&0&0&1&4&10&36&70&235&472&1413&2858&10129&20070\\
        \hline
    \end{array}
\end{displaymath}
\caption{The number $a(n)$ of subgroups of the symmetric group $S_n$ up to conjugacy in $S_n$, and the number $b(n)$ of nonabelian subgroups of $S_n$ up to conjugacy.}\label{Tb:Subgroups}
\end{table}

The first step of our algorithm therefore consists of a determination of subgroups of the symmetric group $S_n$ up to conjugation. This is a nontrivial task. Fortunately, \texttt{GAP} \cite{GAP} can determine these subgroups for $n\le 12$ in a matter of minutes and for $n=13$ in a matter of a few hours. The results are summarized in Table \ref{Tb:Subgroups}. The state of the art results in this area are due to Holt \cite{Holt} who counted (but did not list) subgroups of $S_n$ for $n\le 18$, both absolutely and up to conjugation in $S_n$. To illustrate Holt's results, there are $7598016157515302757$ subgroups of $S_{18}$ partitioned into $7274651$ conjugacy classes.

\bigskip

In Section \ref{Sc:Representation} we prove that there is a one-to-one correspondence between racks (resp. quandles) defined on $X$ and so-called rack envelopes (resp. quandle envelopes) defined on $X$. A \emph{rack envelope} (resp. \emph{quandle envelope}) is a tuple
\begin{displaymath}
    (G,(\lambda_x:x\in X/G))
\end{displaymath}
such that $G\le S_X$, $\lambda_x\in C_G(G_x)$ (resp. $\lambda_x\in Z(G_x)$) for every $x\in X/G$, and $\langle \bigcup_{x\in X/G}\lambda_x^G\rangle = G$. If the last condition is dropped, we speak of rack and quandle folders. It does not seem to be easy to determine without an explicit check if a given folder is in fact an envelope, greatly complicating our enumeration.

The main idea behind rack and quandle envelopes is due to Blackburn \cite[Section 2]{Blackburn} who attributed it in part to Joyce. A special case of quandle envelopes for connected quandles was described in \cite{HSV}, where the ``envelope'' terminology for quandles originated. Although we have independently rediscovered Blackburn's representation of racks and quandles while working on this project, the full credit for the results of Subsection \ref{Ss:Envelopes} should go to Blackburn \cite{Blackburn}.

The isomorphism problem for rack and quandle envelopes is solved in Subsection \ref{Ss:IsomorphismProblem} in terms of an explicit group action \eqref{Eq:OnTuples} of the normalizer $N_{S_X}(G)$. The same group also acts on rack and quandle folders, where the orbits of the action are easier to understand (see Section \ref{Sc:Burnside}).

For any $G\le S_X$ the set of rack/quandle folders $(G,(\lambda_x:x\in X/G))$ is nonempty, containing at least the trivial folder with $\lambda_x=1$ for every $x\in X/G$. Call $G\le S_X$ \emph{rack/quandle admissible} if the set of rack/quandle envelopes over $G$ is also nonempty. Equivalently, $G\le S_X$ is rack/quandle admissible if and only if there is a rack/quandle $(X,*)$ such that $\lmlt{X,*}=G$. We touch upon the question ``Which subgroups $G\le S_X$ are rack/quandle admissible?'' in Subsection \ref{Ss:Admissible} but further investigation would be of considerable interest.

\bigskip

Our algorithm is presented in Section \ref{Sc:Algorithm}, first in a simplified form in Subsection \ref{Ss:BasicAlgorithm} and then with essential improvements. For racks, given a subgroup $G$ of $S_n$, the algorithm returns orbit representatives of the action \eqref{Eq:OnTuples} of the normalizer $N_{S_n}(G)$ on the parameter space
\begin{displaymath}
    \mathrm{Fol}_r(G) = \prod_{x\in X/G} C_G(G_x)
\end{displaymath}
consisting of rack folders, disqualifying those folders that are not envelopes. For quandles, we use quandle folders
\begin{displaymath}
    \mathrm{Fol}_q(G) = \prod_{x\in X/G} Z(G_x)
\end{displaymath}
as the parameter space.

The main obstacle in the algorithm is the fact that the parameter spaces can be quite large. For instance, there exists a subgroup $G$ of $S_{12}$ isomorphic to an elementary abelian $2$-group for which $\mathrm{Fol}_q(G)$ has over $1$ billion elements. The default orbit-stabilizer theorem of \texttt{GAP} cannot cope with spaces this large since it first attempts to convert the action into a permutation action. Nevertheless, it is possible to determine the orbits by employing careful indexing of the space of rack/quandle folders, using one bit of memory for every folder. This is described in Subsection \ref{Ss:Indexing}.

We can further take advantage of the indexing (and other improvements) to minimize storage space for the library of racks and quandles. For instance, the library of $36265070$ racks of order $11$ is stored in a compressed file of approximately 5.7 megabytes, with each rack requiring only $1.25$ bits of storage on average. More details are given in the actual library.

The action \eqref{Eq:OnTuples} is of local character in the following sense. If $f\in N_{S_X}(G)$ and
\begin{displaymath}
    (\kappa_x:x\in X/G) = (\lambda_x:x\in X/G)f,
\end{displaymath}
a given $\kappa_x$ can be calculated once a single $\lambda_z$ is known, namely the $\lambda_z$ with $z$ in the same orbit of $G$ as $xf^{-1}$. This observation is exploited in Subsection \ref{Ss:Precalculated}, where we show how to precalculate the action to crucially speed up the algorithm.

\bigskip

The algorithm is powerful enough to construct all isomorphism types of racks of order $n\le 8$ and quandles of order $n\le 9$ in a matter of seconds, verifying the counts reported in \cite{OEIS}. It takes about a day to determine isomorphism types of racks of order $11$ and about three days to determine isomorphism types of quandles of order $12$.

The algorithm spends most of its running time dealing with elementary abelian $2$-groups contained in $S_n$ and it is finally overwhelmed while trying to determine $r(12)$ and/or $q(13)$. Fortunately, as far as counting of racks and quandles is concerned, all abelian subgroups of $S_n$ can be excluded from the search since they yield precisely $2$-reductive racks and $2$-reductive quandles, which can be counted more efficiently by the methods of \cite{JPSZ} using a different representation and Burnside's Lemma. As we have already mentioned, Jedli\v{c}ka provided us with the numbers $r_{\textrm{$2$-red}}(n)$ and $q_{\textrm{$2$-red}}(n)$ for $n\le 14$. (We have independently verified $r_{\textrm{$2$-red}}(n)$ for $n\le 11$ and $q_{\textrm{$2$-red}}(n)$ for $n\le 12$.)

To determine $r(12)$, $r(13)$ and $q(13)$, we therefore consider only nonabelian subgroups of symmetric groups and obtain $r_{\textrm{non-$2$-red}}(12)$, $r_{\textrm{non-$2$-red}}(13)$ and $q_{\textrm{non-$2$-red}}(13)$ together with the corresponding representatives of isomorphism types. Memory management remains important even in the nonabelian case. For instance, there is a nonabelian subgroup $G$ of $S_{13}$ for which $\mathrm{Fol}_r(G)$ has over $2$ billion elements. It took about two weeks of computing time to determine the most difficult case, $r_{\textrm{non-$2$-red}}(13)$, and thus $r(13)$.

To determine $r_{\textrm{med}}(n)$, we explicitly construct all non-$2$-reductive racks $(X,*)$ of order $n$ and count only those with $\mathrm{Dis}(X,*)$ abelian. We add this count to $r_{\textrm{$2$-red}}(n)$ since every $2$-reductive rack is medial. We proceed similarly for medial quandles.

\bigskip

In Section \ref{Sc:Burnside} we show how to efficiently count the orbits of $N_{S_X}(G)$ on the space of rack or quandle folders. The problem that we really need to solve, namely efficiently counting orbits of $N_{S_X}(G)$ on the space of rack or quandle envelopes, remains open.

Recall that if a finite group $F$ acts on a set $Y$ and $\mathrm{Fix}(Y,f)=\{y\in Y:yf=y\}$ is the set of $f$-invariant elements of $Y$, then Burnside's Lemma states that
\begin{displaymath}
    |Y/F| = \frac{1}{|F|}\sum_{f\in F}|\mathrm{Fix}(Y,f)|.
\end{displaymath}
In our setting we let $F=N_{S_X}(G)$ and $Y=\mathrm{Fol}_r(G)$, the quandle case being similar. For every $f\in N_{S_X}(G)$ we then construct a certain $|X/G|$-partite digraph $\Gamma_r(G,f)$ that encodes the action of $\langle f\rangle$ on $Y$. We describe the structure of $\Gamma_r(G,f)$ in detail and prove that the elements of $\mathrm{Fix}(Y,f)$ are in one-to-one correspondence with unions of certain short directed cycles of $\Gamma_r(G,f)$. These short cycles can be easily counted as long as $X$ is not too large.

\bigskip

The paper concludes with several open problems.

\section{Isomorphisms and conjugation for left quasigroups}\label{Sc:Conjugacy}

For a subgroup $G$ of $S_X$ let $\mathcal L(G)$ be the set of all left quasigroups defined on $X$ whose left multiplication group is equal to $G$.

The following result is likely well-known and it applies to the special cases of racks and quandles:

\begin{proposition}\label{Pr:Conjugation}
Let $X$ be a set and $(X,*)$, $(X,\circ)$ left quasigroups. Then:
\begin{enumerate}
\item[(i)] A bijection $f\in S_X$ is an isomorphism $f:(X,*)\to (X,\circ)$ if and only if $L^\circ_{xf} = (L^*_x)^f$ for every $x\in X$.
\item[(ii)] If $f:(X,*)\to (X,\circ)$ is an isomorphism, then $\lmlt{X,\circ} = (\lmlt{X,*})^f$.
\item[(iii)] If $G$, $H$ are subgroups of $S_X$ that are not conjugate, then no left quasigroup in $\mathcal L(G)$ is isomorphic to a left quasigroup in $\mathcal L(H)$.
\item[(iv)] If $G$, $H$ are conjugate subgroups of $S_X$, then $\mathcal L(G)$ and $\mathcal L(H)$ contain the same isomorphism types of left quasigroups.
\end{enumerate}
\end{proposition}

\begin{proof}
(i) We have $xf\circ yf = (x*y)f$ for every $x$, $y\in X$ if and only if $yL_{xf}^\circ = xf\circ y = (x*yf^{-1})f = yf^{-1}L_x^*f = y(L_x^*)^f$ for every $x$, $y\in X$.

(ii) Using (i), we have $\lmlt{X,\circ} = \langle L^\circ_x:x\in X\rangle = \langle (L^*_{xf^{-1}})^f:x\in X\rangle = \langle L^*_{xf^{-1}}:x\in X\rangle^f = \langle L^*_x:x\in X\rangle^f = (\lmlt{X,*})^f$. Part (iii) now follows, too.

For (iv), suppose that $H=G^f$ for some $f\in S_X$ and let $(X,*)\in\mathcal L(G)$. Then $f:(X,*)\to (X,\circ)$ is an isomorphism, where $(X,\circ)$ is defined by $x\circ y = (xf^{-1}*yf^{-1})f$. Since $\lmlt{X,\circ} = (\lmlt{X,*})^f = G^f = H$ by (ii), we have $(X,\circ)\in\mathcal L(H)$. This shows $\mathcal L(G)\subseteq\mathcal L(H)$ and the other inclusion is proved analogously.
\end{proof}

Consequently, to classify left quasigroups defined on $X$ up to isomorphism, it suffices to:
\begin{itemize}
\item calculate subgroups of $S_X$ up to conjugacy,
\item for each such subgroup, $G$, determine the isomorphism types in $\mathcal L(G)$,
\item return the (necessarily disjoint) union of the isomorphism types.
\end{itemize}

Proposition \ref{Pr:Conjugation} immediately implies:

\begin{corollary}\label{Cr:Iso}
Let $(X,*)$, $(X,\circ)$ be left quasigroups such that $\lmlt{X,*}=\lmlt{X,\circ}=G$. Then every isomorphism $f:(X,*)\to (X,\circ)$ satisfies $f\in N_{S_X}(G)$.
\end{corollary}

\section{Rack and quandle envelopes up to isomorphism}\label{Sc:Representation}

In this section we obtain a one-to-one correspondence between racks and quandles defined on $X$ and certain configurations in $S_X$, called rack and quandle envelopes. We also solve the isomorphism problem for envelopes.

\subsection{Rack and quandle envelopes}\label{Ss:Envelopes}

For a rack $(X,*)$ with $\lmlt{X,*}=G$ let
\begin{displaymath}
    \mathbf E(X,*) = (G,(L_x:x\in X/G)).
\end{displaymath}

\begin{lemma}\label{Lm:ToEnvelope}
Let $(X,*)$ be a rack and $G=\lmlt{X,*}$. Then $(X,*)$ is determined by $\mathbf E(X,*)$. Moreover, $L_x\in C_G(G_x)$ for every $x\in X$ and $G=\langle\bigcup_{x\in X/G} L_x^G\rangle$. If $(X,*)$ is a quandle then $L_x\in Z(G_x)$ for every $x\in X$.
\end{lemma}
\begin{proof}
If $x\in X/G$ and $y\in xG$ are given, let $g\in G$ be such that $xg = y$. Since $g\in G=\lmlt{X,*}\le\aut{X,*}$, we have $(L_x)^g = L_{xg} = L_y$. Hence $(X,*)$ is determined by $\mathbf E(X,*)$. Moreover, $L_x\in C_G(G_x)$ since for any $g\in G_x$ we have $(L_x)^g = L_{xg} =L_x$. Finally, $G=\lmlt{X,*} = \langle L_x:x\in X\rangle = \langle (L_x)^g:x\in X/G,\,g\in G\rangle = \langle\bigcup_{x\in X/G}L_x^G\rangle$. If $(X,*)$ is a quandle, we have also $L_x\in G_x$ (since $x*x=x$) and thus $L_x\in Z(G_x)$.
\end{proof}

\begin{definition}
Let $G\le S_X$. Then $(G,(\lambda_x:x\in X/G))$ is a \emph{rack folder} (resp. \emph{quandle folder}) if $\lambda_x\in C_G(G_x)$ (resp. $\lambda_x\in Z(G_x)$) for every $x\in X/G$. A rack folder (resp. quandle folder) is a \emph{rack envelope} (resp. \emph{quandle envelope}) if $\langle \bigcup_{x\in X/G}\lambda_x^G\rangle = G$.
\end{definition}

Given $G\le S_X$ and $\Lambda=(\lambda_x\in G:x\in X/G)$, we attempt to define a groupoid
\begin{displaymath}
    \mathbf R(G,\Lambda)=(X,*)
\end{displaymath}
by setting
\begin{displaymath}
    L_y = (\lambda_x)^{g_y},
\end{displaymath}
where $x\in X/G$, $y\in xG$ and $g_y$ is any element of $G$ such that $xg_y=y$.

\begin{proposition}\label{Pr:Folder}
Let $G\le S_X$, $\Lambda = (\lambda_x\in G:x\in X/G)$ and $(X,*)=\mathbf R(G,\Lambda)$. Then:
\begin{enumerate}
\item[(i)] $(X,*)$ is well-defined if and only if $(G,\Lambda)$ is a rack folder, in which case $(X,*)$ is a rack satisfying $L_x=\lambda_x$ for every $x\in X/G$ and $\lmlt{X,*} = \langle \bigcup_{x\in X/G}\lambda_x^G\rangle\le G$,
\item[(ii)] $(X,*)$ is a well-defined quandle if and only if $(G,\Lambda)$ is a quandle folder.
\end{enumerate}
\end{proposition}
\begin{proof}
(i) Suppose that $(X,*)$ is well-defined and let $x\in X/G$. For every $g_x\in G_x$ we have $\lambda_x^{g_x} = L_x = \lambda_x$, where the latter equality follows by taking $g_x=1$. Thus $\lambda_x \in C_G(G_x)$. Conversely, suppose that $\lambda_x\in C_G(G_x)$ holds for every $x\in X/G$ and let $g$, $h\in G$ be such that $xg=xh$. Since $gh^{-1}\in G_x$, we have $\lambda_x^{gh^{-1}}=\lambda_x$, that is, $\lambda_x^g = \lambda_x^h$, and $(X,*)$ is well-defined.

Now suppose that $(X,*)$ is well-defined, necessarily a left quasigroup. We claim that $(X,*)$ is a rack. Fix $u$, $v$, $w\in X$ and let $x\in X/G$, $g_v\in G$ be such that $xg_v=v$. Since $g_vL_u\in G$ satisfies $xg_vL_u = vL_u = u*v$, we have $L_{u*v} = \lambda_x^{g_vL_u} = (\lambda_x^{g_v})^{L_u} = L_v^{L_u}$ and hence
\begin{displaymath}
    (u*v)*(u*w) = wL_uL_{u*v} = wL_uL_v^{L_u} = wL_vL_u = u*(v*w).
\end{displaymath}
We certainly have $\lmlt{X,*} = \langle \bigcup_{x\in X/G}L_x^G\rangle = \langle \bigcup_{x\in X/G}\lambda_x^G\rangle \le G$.

(ii) If $(X,*)$ is a well-defined quandle then $(G,\Lambda)$ is a rack folder by (i) and for every $x\in X/G$ we have $x\lambda_x = x$, that is, $\lambda_x\in C_G(G_x)\cap G_x = Z(G_x)$. Conversely, if $(G,\Lambda)$ is a quandle folder then it is a rack folder, $(X,*)$ is a rack by (i), and for every $x\in X/G$, $y\in xG$ and $g_y\in G$ such that $xg_y=y$ we have $yL_y = y\lambda_x^{g_y} = yg_y^{-1}\lambda_x g_y = x\lambda_x g_y = xg_y = y$, where we have used $\lambda_x\in G_x$ in the penultimate step.
\end{proof}

\begin{theorem}[Correspondence between racks/quandles and rack/quandle envelopes]\label{Th:Correspondence}
Let $X$ be a nonempty set and $G\le S_X$. There is a one-to-one correspondence between racks/quandles $(X,*)$ satisfying $\lmlt{X,*}=G$ and rack/quandle envelopes $(G,\Lambda)$. Given a rack/quandle $(X,*)$ with $\lmlt{X,*}=G$, the corresponding rack/quandle envelope is $\mathbf E(X,*)$. Given a rack/quandle envelope $(G,\Lambda)$, the corresponding rack/quandle is $\mathbf R(G,\Lambda)$.
\end{theorem}
\begin{proof}
Suppose that $(X,*)$ is a rack/quandle with $\lmlt{X,*}=G$. By Lemma \ref{Lm:ToEnvelope}, $\mathbf E(X,*) = (G,(L_x^*:x\in X/G))$ is a rack/quandle envelope and $G=\langle\bigcup_{x\in X/G}(L_x^*)^G\rangle$. By Proposition \ref{Pr:Folder}, $(X,\circ)=\mathbf R(\mathbf E(X,*))$ is a rack/quandle satisfying $L_x^\circ = L_x^*$ for every $x\in X/G$ and $\lmlt{X,\circ} = \langle\bigcup_{x\in X/G}(L_x^*)^G\rangle = G$. Lemma \ref{Lm:ToEnvelope} then implies that $(X,*)=(X,\circ)$.

Conversely, suppose that $\Lambda=(\lambda_x:x\in X/G)$ and $(G,\Lambda)$ is a rack/quandle envelope. By Proposition \ref{Pr:Folder}, $(X,*)=\mathbf R(G,\Lambda)$ is a rack/quandle satisfying $L_x^*=\lambda_x$ for every $x\in X/G$ and $\lmlt{X,*}=\langle \bigcup_{x\in X/G}\lambda_x^G\rangle = G$, where the last equality holds because $(G,\Lambda)$ is an envelope. Finally, we have $\mathbf E(\mathbf R(G,\Lambda)) = \mathbf E(X,*) = (G,(L_x^*:x\in X/G)) = (G,\Lambda)$.
\end{proof}

Note that we do not claim that there is a one-to-one correspondence between rack (or quandle) folders $(G,\Lambda)$ and racks (or quandles) $(X,*)$ satisfying $\lmlt{X,*}\le G$. Indeed, if $(G,\Lambda)$ is a rack folder  then there might exist several racks $(X,*)$ such that $L_x^*=\lambda_x$ for every $x\in X/G$ and $\lmlt{X,*}\le G$. (Let $G=S_X$, $X/G=\{x_0\}$ and $\lambda_{x_0}=1$. Then $(G,\Lambda)$ is a rack folder and any rack $(X,*)$ satisfying $L_{x_0}^*=1$ does the job.) Conversely, if $(X,*)$ is a rack such that $\lmlt{X,*}\le G$, it might not be the case that $L_x^*\in C_G(G_x)$ for every $x\in X/G$. (Consider a rack with some $L_x^*\ne 1$ but take $G=S_X$ with $|X|$ large enough so that $C_G(G_x)=1$.)

\subsection{Envelopes up to isomorphism}\label{Ss:IsomorphismProblem}

We present a solution to the isomorphism problem for rack and quandle envelopes. Thanks to Proposition \ref{Pr:Conjugation}, it suffices to consider the case when the two corresponding racks have the same left multiplication groups.

\begin{proposition}\label{Pr:Iso}
Let $(G,(\lambda_x:x\in X/G))$ and $(G,(\kappa_x:x\in X/G))$ be rack/quandle envelopes. For every $x\in X/G$ and $y\in xG$ let $g_y\in G$ be such that $xg_y=y$. Then the corresponding racks/quandles are isomorphic if and only if there is $f\in N_{S_X}(G)$ such that
\begin{equation}\label{Eq:Action}
    \kappa_x = ((\lambda_{yg_y^{-1}})^{g_y})^f
\end{equation}
for every $x\in X/G$, where $y = xf^{-1}$.
\end{proposition}

\begin{proof}
Let $(X,*)$ and $(X,\circ)$ be the racks/quandles corresponding to $(G,(\lambda_x:x\in X/G))$ and $(G,(\kappa_x:x\in X/G))$, respectively. By Corollary \ref{Cr:Iso}, $(X,*)$ is isomorphic to $(X,\circ)$ if and only if there is an isomorphism $f:(X,*)\to (X,\circ)$ such that $f\in N_{S_X}(G)$. By Proposition \ref{Pr:Conjugation}, $f$ is an isomorphism if and only if $L_{xf}^\circ = (L_x^*)^f$ for every $x\in X$, or, equivalently, $L_x^\circ = (L_{xf^{-1}}^*)^f$ for every $x\in X$. In fact, since $(X,\circ)$ is determined by $G$ and the left translations $L_x^\circ$ with $x\in X/G$ (see Lemma \ref{Lm:ToEnvelope}), the last condition can be equivalently restated as $L_x^\circ = (L_{xf^{-1}}^*)^f$ for every $x\in X/G$.

Let $x\in X/G$. We certainly have $\kappa_x = L_x^\circ$. Now, $y = xf^{-1}$ is not necessarily in $X/G$, but $yg_y^{-1}$ is an element of $X/G$ (possibly distinct from $x$), so
\begin{displaymath}
   (L^*_{xf^{-1}})^f = (L^*_y)^f = ((L^*_{yg_y^{-1}})^{g_y})^f = ((\lambda_{yg_y^{-1}})^{g_y})^f,
\end{displaymath}
finishing the proof.
\end{proof}

For $G\le S_X$, let
\begin{displaymath}
    \mathrm{Fol}_r(G),\quad \mathrm{Fol}_q(G),\quad \mathrm{Env}_r(G), \quad \mathrm{Env}_q(G)
\end{displaymath}
be, respectively, the sets of all rack folders, quandle folders, rack envelopes and quandle envelopes on $X$ of the form $(G,\Lambda)$. Given $f\in N_{S_X}(G)$ and an element $(\lambda_x:x\in X/G)$ of one of the above spaces, we define
\begin{equation}\label{Eq:OnTuples}
    (\lambda_x:x\in X/G)f = (\kappa_x:x\in X/G),
\end{equation}
where for every $x\in X/G$ the bijection $\kappa_x$ is obtained by \eqref{Eq:Action}.

\begin{theorem}[Rack/quandle envelopes up to isomorphism]\label{Th:Action}
Let $X$ be a nonempty set, $G\le S_X$ and $F=N_{S_X}(G)$. Then:
\begin{enumerate}
\item[(i)] The group $F$ acts on each of $\mathrm{Fol}_r(G)$, $\mathrm{Fol}_q(G)$, $\mathrm{Env}_r(G)$ and $\mathrm{Env}_q(G)$ via \eqref{Eq:OnTuples}.
\item[(ii)] The orbits of $F$ on $\mathrm{Env}_r(G)$ (resp. $\mathrm{Env}_q(G)$) are in one-to-one correspondence with isomorphism types of racks (resp. quandles) defined on $X$ with left multiplication groups equal to $G$.
\item[(iii)] If an orbit of $F$ on $\mathrm{Fol}_r(G)$ (resp. $\mathrm{Fol}_q(G)$) contains an element of $\mathrm{Env}_r(G)$ (resp. $\mathrm{Env}_q(G)$), then the entire orbit is a subset of $\mathrm{Env}_r(G)$ (resp. $\mathrm{Env}_q(G)$).
\end{enumerate}
\end{theorem}
\begin{proof}
Proposition \ref{Pr:Iso} settles part (ii) and also part (i) for the case of rack and quandle envelopes. Does $F$ act on rack/quandle folders? Suppose that $(G,\Lambda)$ is a rack/quandle folder and let $(X,*)=\mathbf R(G,\Lambda)$. Let $f\in F$ and let $(X,\circ)$ be such that $f:(X,*)\to (X,\circ)$ is an isomorphism. Using the same notation as in the proof of Proposition \ref{Pr:Iso}, we have $L^\circ_x = (L^*_{xf^{-1}})^f = ((\lambda_{yg_y^{-1}})^{g_y})^f$, which means that $(\lambda_x:x\in X/G)f = (L^\circ_x:x\in X/G)$. By Proposition \ref{Pr:Folder}, $(G,(L^\circ_x:x\in X/G))$ is a rack/quandle folder. This proves (i). Part (iii) follows.
\end{proof}

%Note that we do not claim that the orbits of $N_{S_X}(G)$ on $\mathrm{Fol}_r(G)$ correspond to isomorphism types of racks whose left multiplication group is contained in $G$. There are at least two reasons for this. First, if $\Lambda\in \mathrm{Fol}_r(G)\setminus\mathrm{Env}_r(G)$ then the canonically constructed rack $\mathbf R(G,\Lambda)$ of Section \ref{Sc:Representation} indeed satisfies $\lmlt{\mathbf R(G,\Lambda)}\le G$, but it is not necessarily the only rack $(X,*)$ up to isomorphism with the property $\lmlt{X,*}\le G$ and $\lambda_x=L_x$ for every $x\in X/G$. Conversely, if $(X,*)$ is a rack with $\lmlt{X,*}=H<G$, then we certainly have $L_x\in C_H(H_x)$ for every $x\in X$, but not necessarily $L_x\in C_G(G_x)$.

\subsection{Remarks on rack and quandle admissibility}\label{Ss:Admissible}

A subgroup $G\le S_X$ is said to be \emph{rack admissible} (resp. \emph{quandle admissible}) if there is a rack (resp. quandle) $(X,*)$ such that $\lmlt{X,*}=G$. Note that it is necessary to keep track of the way $G$ acts on $X$, not just of the isomorphism type of $G$.

\begin{proposition}\label{Pr:Admissible}
Let $G\le S_X$. If $G$ is rack admissible then $\bigcup_{x\in X/G} (C_G(G_x))^G$ generates $G$. If $G$ is quandle admissible then $\bigcup_{x\in X/G} Z(G_x)^G$ generates $G$.
\end{proposition}

\begin{proof}
If $G$ is rack admissible then by Theorem \ref{Th:Correspondence} there is a rack envelope $(G,(\lambda_x:x\in G/X))$. Since $\lambda_x\in C_G(G_x)$ for every $x\in X/G$, we have $G = \langle \bigcup_{x\in X/G}\lambda_x^G\rangle \le \langle \bigcup_{x\in X/G} (C_G(G_x))^G\rangle\le G$. The quandle case is similar.
\end{proof}

The necessary condition of Proposition \ref{Pr:Admissible} disqualifies many ``large'' subgroups of $S_X$ from the search for racks and quandles. Certainly $S_X$ itself is disqualified:

\begin{corollary}
Let $|X|\ge 4$. Then $S_X$ as a subgroup of itself is not rack admissible.
\end{corollary}
\begin{proof}
The group $G=S_X$ acts transitively on $X$. For $x\in X$, $G_x$ is isomorphic to $S_{X\setminus\{x\}}$. Since $|X|\ge 4$, we have $C_G(G_x) \cong C_{S_X}(S_{X\setminus\{x\}}) = 1$ and hence $(C_G(G_x))^G=1$ does not generate $G$.
\end{proof}

It is well-known that every rack (resp. quandle) of order $n$ embeds into a rack (resp. quandle) of order $n+1$, proving that both $r$ and $q$ are non-decreasing functions. We give a short argument based on envelopes. It is certainly possible to give an elementary proof on the level of multiplication tables.

\begin{proposition}
Let $X$ be a set and $z$ an element not contained in $X$. Every rack/quandle on $X$ embeds into a rack/quandle on $X\cup\{z\}$.
\end{proposition}

\begin{proof}
Let $\overline{X}=X\cup\{z\}$. For $\sigma\in S_X$ define $\overline{\sigma}\in S_{\overline{X}}$ by $x\overline{\sigma}=x$ if $x\in X$ and $z\overline{\sigma}=z$. Let $(G,(\lambda_x:x\in X/G))$ be a rack envelope. The orbits of $\overline{G}$ are the same as those of $G$, except for the additional singleton orbit $\{z\}$. For $x\in X/G$ let $\kappa_x = \overline{\lambda_x}$ and observe that $\kappa_x\in C_{\overline{G}}(\overline{G}_x)$ because $\lambda_x\in C_G(G_x)$. Let $\kappa_z$ be the identity on $\overline{X}=X\cup\{z\}$, clearly satisfying $\kappa_z\in C_{\overline{G}}(\overline{G}_z)$. Then $(\overline{G},(\kappa_x:x\in\overline{X}/\overline{G}))$ is a rack envelope. The quandle case is similar.
\end{proof}

\section{The algorithm}\label{Sc:Algorithm}

\subsection{A basic algorithm}\label{Ss:BasicAlgorithm}

It follows from Proposition \ref{Pr:Conjugation}, Theorem \ref{Th:Correspondence} and Theorem \ref{Th:Action} that the algorithm in Figure \ref{Fg:BasicAlg} returns a complete set of isomorphism types of racks of order $n$ (more precisely, rack envelopes on $X=\{1,\dots,n\}$).

\begin{figure}[ht]
\begin{displaymath}
\begin{array}{l}
\hline
\phantom{x}\\
{\scriptscriptstyle 01}\quad X \leftarrow \{1,\dots,n\}\\
{\scriptscriptstyle 02}\quad \mathcal G \leftarrow \text{subgroups of $S_X$ up to conjugation in $S_X$}\\
{\scriptscriptstyle 03}\quad \texttt{for } G \texttt{ in } \mathcal G \texttt{ do}\\
{\scriptscriptstyle 04}\quad \tab \texttt{if MightBeRackAdmissible}(G,X) \texttt{ then}\\
{\scriptscriptstyle 05}\quad \tab\tab Y\leftarrow \mathrm{Fol}_r(G) = \prod_{x\in X/G} C_G(G_x)\\
{\scriptscriptstyle 06}\quad \tab\tab O\leftarrow \text{orbit representatives of the action of $N_{S_X}(G)$ on $Y$ given by \eqref{Eq:OnTuples}}\\
{\scriptscriptstyle 07}\quad \tab\tab R_G\leftarrow \{(G,(\lambda_x:x\in X/G))\in O: \langle \bigcup_{x\in X/G}\lambda_x^G\rangle = G\}\\
{\scriptscriptstyle 08}\quad \tab\texttt{else}\\
{\scriptscriptstyle 09}\quad \tab\tab R_G\leftarrow\emptyset\\
{\scriptscriptstyle 10}\quad \tab\texttt{end if}\\
{\scriptscriptstyle 11}\quad \texttt{end for}\\
{\scriptscriptstyle 12}\quad \texttt{return } \bigcup_{G\in\mathcal G}R_G\\
\phantom{x}\\
\hline
\end{array}
\end{displaymath}
\caption{A basic algorithm for enumeration of racks and quandles up to isomorphism.}\label{Fg:BasicAlg}
\end{figure}

The function \texttt{MightBeRackAdmissible}$(G,X)$ in line $04$ returns \texttt{true} iff $\langle\bigcup_{x\in X/G}(C_G(G_x))^G\rangle = G$, cf. Proposition \ref{Pr:Admissible}. Note that \texttt{true} might be returned even if $G\le S_X$ is in fact not rack admissible.

Theorem \ref{Th:Action}(iii) guarantees that it is safe to discard the entire orbit of $(G,\Lambda)$ in line $07$ if $(G,\Lambda)$ is not a rack envelope.

For quandles, it suffices to modify the algorithm as follows:
\begin{itemize}
\item replace \texttt{MightBeRackAdmissible}$(G,X)$ in line $04$ with \texttt{MightBeQuandleAdmissible}$(G,X)$, which returns \texttt{true} iff $\langle\bigcup_{x\in X/G}Z(G_x)^G\rangle = G$, and
\item populate the variable $Y$ in line $05$ with $\mathrm{Fol}_q(G) = \prod_{x\in X/G}Z(G_x)$.
\end{itemize}

We remark that the algorithm will struggle when $N_{S_X}(G)$ is large, which tends to happen when $G$ is either very small or very large. The extreme case $G=1$ (which yields $N_{S_X}(G)=S_X$) can be handled separately since then the set $\mathrm{Fol}_r(G)$ is a singleton. Large subgroups of $S_X$ are typically disqualified by failing the necessary condition of Proposition \ref{Pr:Admissible}.

\subsection{Indexing}\label{Ss:Indexing}

In some computational packages it is possible to define the action \eqref{Eq:OnTuples} on the domain $Y$ and then call standard methods for orbit representatives. We encounter two difficulties. The space $Y$ can be too large (more than billion elements for some $G\le S_{12}$) to be stored in memory. Even if $Y$ can be stored in memory, the default methods of convert the action into a permutation action on $\{1,\dots,|Y|\}$ and the algorithm might run out of memory then. In this subsection we will show how to make the algorithm work on larger domains than the default methods would allow. A similar approach to indexing and actions on large domains is implemented in \texttt{GAP}, cf. methods \texttt{PositionCanonical} and \texttt{OrbitStabilizerAlgorithm}.

Let $G\le S_X$. We will efficiently index the space $Y=\prod_{x\in X/G} C_G(G_x)$ of rack folders, the case of quandle folders being similar. (However, it is not easy to index elements of the subset of rack envelopes or quandle envelopes.) Let $<$ be a lexicographical order on $Y$ inherited from an linear order on $S_X$. Let $p(C_G(G_x),\lambda_x)$ be the position of $\lambda_x$ in $C_G(G_x)$ with respect to $<$. Then $\Lambda = (\lambda_x:x\in X/G)\in Y$ can be identified with the numerical vector $(p(C_G(G_x),\lambda_x):x\in X/G)$, which can in turn be identified with an element of the interval $\{1,\dots,|Y|\}$, for instance by using a hybrid-base expansion. It is then easy to implement the conversion functions \texttt{Element}$(Y,i)$ (that returns the $i$th element of $Y$) and \texttt{Position}$(Y,\Lambda)$ (that returns the position of $\Lambda$ in $Y$).

%The following example illustrates the indexing.

%\begin{example}
%Let $X=\{1,\dots,5\}$ and $G=\langle (1,2),(3,4,5),(3,4)\rangle\le S_X$. Then $X/G=\{1,3\}$ and $C_G(G_1) =\{(),(1,2)\}$, $C_G(G_3) = \{(),(4,5),(1,2),(1,2)(4,5)\}$, where we have used the default ordering of elements of $S_X$ in \texttt{GAP}. The $8$ elements of $\mathrm{Fol}_r(G)$, with $G$ suppressed, can then be indexed as follows:
%\begin{displaymath}
%\begin{array}{llll}
%    1=((),()),  &2=((),(4,5)), &3=((),(1,2)), &4=((),(1,2)(4,5)),\\
%    5=((1,2),()), &6=((1,2),(4,5)), &7=((1,2),(1,2)), &8=((1,2),(1,2)(4,5)).
%\end{array}
%\end{displaymath}
%\end{example}

Given $G\le S_X$ and $Y=\mathrm{Fol}_r(G)$, the algorithm in Figure \ref{Fg:ImprovedAlg} returns the desired orbit representatives of the action of $N_{S_X}(G)$ on $Y$; it can be used to replace lines $06$ and $07$ of the algorithm in Figure \ref{Fg:BasicAlg}. Note that thanks to the indexing functions described above we do not need to keep $Y$ in memory, only the much smaller subsets $C_G(G_x)$ and a binary vector of length $|Y|$.

\begin{figure}[h]
\begin{displaymath}
\begin{array}{l}
\hline
\phantom{x}\\
{\scriptscriptstyle 06:01}\quad V\leftarrow \text{binary array of length $|Y|$ with all values initialized to \texttt{true}}\\
{\scriptscriptstyle 06:02}\quad \texttt{for } i \texttt{ in } \{1,\dots,|Y|\} \texttt{ do}\\
{\scriptscriptstyle 06:03}\quad \tab \texttt{if } V[i] = \texttt{true then}\\
{\scriptscriptstyle 06:04}\quad \tab\tab \Lambda = (\lambda_x:x\in X/G) \leftarrow \texttt{Element}(Y,i)\\
{\scriptscriptstyle 06:05}\quad \tab\tab \texttt{if } \langle\bigcup_{x\in X/G}\lambda_x^G\rangle\ne G \texttt{ then}\\
{\scriptscriptstyle 06:06}\quad \tab\tab\tab V[i]\leftarrow \texttt{false}\\
{\scriptscriptstyle 06:07}\quad \tab\tab \texttt{end if}\\
{\scriptscriptstyle 06:08}\quad \tab\tab \texttt{for } f \texttt{ in } N_{S_X}(G) \texttt{ do}\\
{\scriptscriptstyle 06:09}\quad \tab\tab\tab p\leftarrow \texttt{Position}(Y,\Lambda f)\\
{\scriptscriptstyle 06:10}\quad \tab\tab\tab \texttt{if } p>i \texttt{ then}\\
{\scriptscriptstyle 06:11}\quad \tab\tab\tab\tab V[p]\leftarrow \texttt{false}\\
{\scriptscriptstyle 06:12}\quad \tab\tab\tab \texttt{end if}\\
{\scriptscriptstyle 06:13}\quad \tab\tab \texttt{end for}\\
{\scriptscriptstyle 06:14}\quad \tab\texttt{end if}\\
{\scriptscriptstyle 06:15}\quad \texttt{end for}\\
{\scriptscriptstyle 06:16}\quad R_G\leftarrow \{(G,\texttt{Element}(Y,i)):1\le i\le |Y|,\,V[i]=\texttt{true}\}\\
{\scriptscriptstyle 06:17}\quad \texttt{return } R_G\\
\phantom{x}\\
\hline
\end{array}
\end{displaymath}
\caption{An algorithm for orbit representatives of $N_{S_X}(G)$ on $Y=\mathrm{Fol}_r(G)$.}\label{Fg:ImprovedAlg}
\end{figure}

The test in line $06{:}05$ ensures that folders $(G,\Lambda)$ that are not envelopes will not be returned. Note that the value of $p$ in line $06{:}09$ can never be less than $i$.

\subsection{Precalculating the action of $N_{S_X}(G)$}\label{Ss:Precalculated}

It is time consuming to calculate $\Lambda f$ in line $06{:}09$, inside the innermost cycle of the algorithm. In this subsection we will show how the algorithm can be substantially improved by precalculating the action of $N_{S_X}(G)$ on the space $Y$. We will take advantage of the fact that the action \eqref{Eq:OnTuples} can be localized, i.e., when
\begin{displaymath}
    (\kappa_x:x\in X/G) = (\lambda_x:x\in X/G)f,
\end{displaymath}
the value of $\kappa_x$ depends only on $y=xf^{-1}$, $g_y$ and $\lambda_{yg_y^{-1}}$, rather than on the entire folder $(\lambda_x:x\in X/G)$.

Suppose that $G\le S_X$ is given. Given $f\in N_{S_X}(G)$ and $x\in X/G$, let $y = xf^{-1}$ and $z\in (X/G)\cap yG$, and for every $\lambda_z\in C_G(G_z)$ let us precalculate
\begin{displaymath}
    I(f,x,\lambda_z) = p(C_G(G_x),((\lambda_z)^{g_y})^f).
\end{displaymath}
Line $06{:}09$ can then be replaced with the code in Figure \ref{Fg:Position}.

\begin{figure}
\begin{displaymath}
\begin{array}{l}
\hline
\phantom{x}\\
{\scriptscriptstyle 06:09:01}\quad \texttt{for } x\in X/G \texttt{ do}\\
{\scriptscriptstyle 06:09:02}\quad \tab y\leftarrow xf^{-1}\\
{\scriptscriptstyle 06:09:03}\quad \tab z\leftarrow \text{the unique element of $(X/G)\cap yG$}\\
{\scriptscriptstyle 06:09:04}\quad \tab p_x\leftarrow I(f,x,\lambda_z)\\
{\scriptscriptstyle 06:09:05}\quad \texttt{end for}\\
{\scriptscriptstyle 06:09:06}\quad p\leftarrow \text{the index of $(p_x:x\in X/G)$ as an element of $\{1,\dots,|Y|\}$}\\
\phantom{x}\\
\hline
\end{array}
\end{displaymath}
\caption{Code for the position of $(\lambda_x:x\in X/G)f$ in $\mathrm{Fol}_r(G)$ with precalculated action.}\label{Fg:Position}
\end{figure}

We have now finished describing the main features of the algorithm by which we have obtained the isomorphism types of racks (resp. quandles) of orders $\le 11$ (resp. $\le 12$). We have explained in the Introduction how the cases $r(12)$, $r(13)$ and $q(13)$ were handled.

\section{Counting orbits of the action on rack and quandle folders}\label{Sc:Burnside}

In this section we visualize the action \eqref{Eq:OnTuples} of $N_{S_X}(G)$ on $\mathrm{Fol}_r(G)$ and count its orbits, the case of quandle folders being similar. Unfortunately, this approach does not solve the orbit counting problem on the space of rack and quandle envelopes.

Let $G\le S_X$, $F=N_{S_X}(G)$ and $f\in F$. Construct a digraph (possibly with loops) $\Gamma_r(G,f)$ as follows. The vertex set of $\Gamma_r(G,f)$ is the disjoint union of the sets $C_G(G_x)$ for $x\in X/G$. (The sets $C_G(G_x)$ are not disjoint as subsets of $S_X$, but we will treat them as being formally disjoint for the digraph construction.) Given not necessarily distinct $x$, $z\in X/G$ and $\kappa_x\in C_G(G_x)$, $\lambda_z\in C_G(G_z)$, we declare $\lambda_z\to\kappa_x$ to be a directed edge of $\Gamma_r(G,f)$ if and only if with $y=xf^{-1}$ we have $z=yg_y^{-1}$ and $\kappa_x = ((\lambda_z)^{g_y})^f$.

The digraph $\Gamma_r(G,f)$ can be seen as a visualization of the action of $\langle f\rangle$ on $\mathrm{Fol}_r(G)$. The elements $(\lambda_x:x\in X/G)$ of $\mathrm{Fol}_r(G)$ correspond precisely to selections of vertices of $\Gamma_r(G,f)$, one in each vertex set $C_G(G_x)$. When $(\kappa_x:x\in X/G)=(\lambda_x:x\in X/G)f$, then the tuple $(\kappa_x:x\in X/G)$ is obtained from $(\lambda_x:x\in X/G)$ by moving away from every $\lambda_x$ along the (unique) directed edge starting at $\lambda_x$.

Before we describe some properties of $\Gamma_r(G,f)$, let us observe that $f\in N_{S_X}(G)$ induces a permutation of $X/G$. Indeed, suppose that $x$, $y\in G$ belong to the same orbit of $G$, so $x=yg$ for some $g\in G$. Then $xff^{-1}gf=xgf=yf$ and $f^{-1}gf\in G$ show that $xf$, $yf$ belong to the same orbit of $G$.

\begin{proposition}\label{Pr:Digraph}
Let $X$ be a finite set, $G\le S_X$, $f\in N_{S_X}(G)$ and $\Gamma = \Gamma_r(G,f)$. Let $\bar{f}$ be the permutation of $X/G$ induced by $f$. Then:
\begin{enumerate}
\item[(i)] $\Gamma$ is an $|X/G|$-partite digraph with parts $\{C_G(G_x):x\in X/G\}$.
\item[(ii)] For $x$, $z\in X/G$, there exists an edge from $C_G(G_z)$ to $C_G(G_x)$ if and only if $\bar{f}$ maps $zG$ to $xG$.
\item[(iii)] The digraph induced by $\Gamma$ on $X/G$ by collapsing every vertex set $C_G(G_x)$ into a single vertex is a disjoint union of directed cycles, namely the cycle decomposition of $\bar{f}$.
\item[(iv)] Every vertex of $\Gamma$ has indegree equal to $1$ and outdegree equal to $1$, so $\Gamma$ is a union of disjoint directed cycles.
\end{enumerate}
\end{proposition}
\begin{proof}
Part (i) follows from the definition of $\Gamma$. For (ii), note that the identity permutation $1$ is present in every $C_G(G_x)$ and if $y=xf^{-1}$, $z=yg_y^{-1}$, $\kappa_x=1$ and $\lambda_z=1$, then $\kappa_x = ((\lambda_z)^{g_y})^f$. Part (iii) follows.

For (iv), suppose that $\lambda_z\to \kappa_x$ and $\mu_z\to\kappa_x$ are edges and let $y=xf^{-1}$ as usual. Then $((\lambda_z)^{g_y})^f = \kappa_x = ((\mu_z)^{g_y})^f$ and therefore $\lambda_z=\mu_z$. Dually, if $\lambda_z\to\kappa_x$ and $\lambda_z\to\nu_x$ are edges, then $\kappa_x = ((\lambda_z)^{g_y})^f = \nu_x$. Since the indegree and outdegree of every vertex is equal to $1$, $\Gamma$ is a disjoint union of directed cycles.
\end{proof}

Let us illustrate the digraph construction with two small examples.

\begin{example}\label{Ex:LongCycle}
Let $X=\{1,\dots,5\}$ and let $G=\langle (1,2)(3,4,5)\rangle\cong C_6$ be a subgroup of $S_X$. Then $G$ has orbits $\{1,2\}$, $\{3,4,5\}$ and we can take $X/G=\{1,3\}$. We have $C_G(G_1) = C_G(G_3) = G$ and $F=N_{S_5}(G) = \langle G,(4,5)\rangle$. Consider $f=(1,2)(4,5)\in F$. The permutation $\bar{f}$ on $X/G$ induced by $f$ is trivial and the directed graph $\Gamma_r(G,f)$ is as follows:
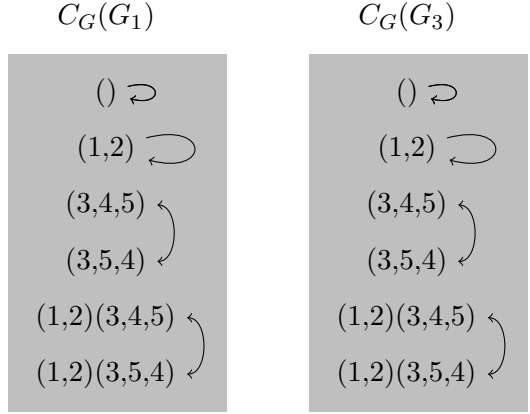
\begin{displaymath}
\begin{tikzpicture}
\node (x1) at (0,1) {$C_G(G_1)$};
\node (x3) at (4,1) {$C_G(G_3)$};
\fill [lightgray] (-1.3,-4.25) rectangle (1.6,0.5);
\fill [lightgray] (2.7,-4.25) rectangle (5.6,0.5);
\node (a1) at (0,0) {()};
\node (b1) at (0,-0.75) {(1,2)};
\node (c1) at (0,-1.5) {(3,4,5)};
\node (d1) at (0,-2.25) {(3,5,4)};
\node (e1) at (0,-3) {(1,2)(3,4,5)};
\node (f1) at (0,-3.75) {(1,2)(3,5,4)};

\node (a3) at (4,0) {()};
\node (b3) at (4,-0.75) {(1,2)};
\node (c3) at (4,-1.5) {(3,4,5)};
\node (d3) at (4,-2.25) {(3,5,4)};
\node (e3) at (4,-3) {(1,2)(3,4,5)};
\node (f3) at (4,-3.75) {(1,2)(3,5,4)};

\draw[->] (a1) edge[loop right] (a1);
\draw[->] (b1) edge[loop right] (b1);
\draw[->] (a1) edge[loop right] (a1);
\draw[<->] (c1) edge[out=0, in=0] (d1);
\draw[<->] (e1) edge[out=0, in=0] (f1);

\draw[->] (a3) edge[loop right] (a3);
\draw[->] (b3) edge[loop right] (b3);
\draw[->] (a3) edge[loop right] (a3);
\draw[<->] (c3) edge[out=0, in=0] (d3);
\draw[<->] (e3) edge[out=0, in=0] (f3);
\end{tikzpicture}
\end{displaymath}
For instance, there is a directed edge $(3,4,5)\to (3,5,4)$ in the vertex set $C_G(G_1)$ because with $x=1$, $y=xf^{-1}=2$ and $g_y=(1,2)$ we have $((3,4,5)^{g_y})^f = (3,4,5)^f = (3,5,4)$. (Since $G$ is abelian here, the conjugations by the $g_y$s can be omitted.)
\end{example}

\begin{example}
Let $X=\{1,\dots,7\}$ and let $G=\langle (1,2),(1,2,3),(4,5),(4,5,6)\rangle\cong S_3\times S_3$ be a subgroup of $S_X$. Then $G$ has orbits $\{1,2,3\}$, $\{4,5,6\}$, $\{7\}$ and we can take $X/G=\{1,4,7\}$. We have $C_G(G_1) = \langle (2,3)\rangle$, $C_G(G_4) = \langle (5,6)\rangle$, $C_G(G_7) = 1$ and $F=N_{S_7}(G) = \langle G,(1,4)(2,5)(3,6)\rangle$. Consider $f=(1,5)(2,4)(3,6)\in F$. The permutation of $X/G$ induced by $f$ is given by $(1,4)(7)$ and the directed graph $\Gamma_r(G,f)$ is as follows:
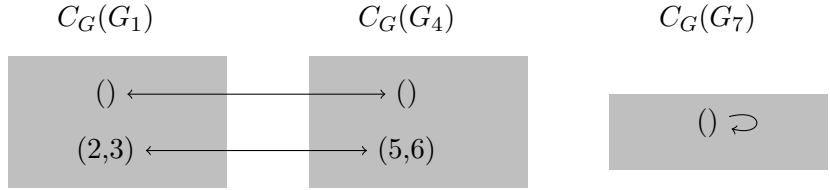
\begin{displaymath}
\begin{tikzpicture}
\node (x1) at (0,1) {$C_G(G_1)$};
\node (x4) at (4,1) {$C_G(G_4)$};
\node (x7) at (8,1) {$C_G(G_7)$};
\fill [lightgray] (-1.3,-1.25) rectangle (1.6,0.5);
\fill [lightgray] (2.7,-1.25) rectangle (5.6,0.5);
\fill [lightgray] (6.7,-1) rectangle (9.6,0);
\node (a1) at (0,0) {()};
\node (b1) at (0,-0.75) {(2,3)};
\node (a4) at (4,0) {()};
\node (b4) at (4,-0.75) {(5,6)};
\node (a7) at (8,-0.375) {()};
\draw[<->] (a1) -- (a4);
\draw[<->] (b1) -- (b4);
\draw[->] (a7) edge[loop right] (a7);
\end{tikzpicture}
\end{displaymath}
This must be the case because the cycle structure of permutations is preserved by conjugation. To give a detailed calculation in one case, there is an edge from $(5,6)$ in $C_G(G_4)$ to $(2,3)$ in $C_G(G_1)$ since with $x=1$, $y=xf^{-1}=5$ and $g_y=(4,5)$ we have $((5,6)^{g_y})^f = (2,3)$.
\end{example}

Note that a directed cycle of $\Gamma_r(G,f)$ can visit a given vertex set $C_G(G_x)$ more than once before closing upon itself. For instance, in Example \ref{Ex:LongCycle}, there is a cycle that starts in $C_G(G_1)$ at $(3,4,5)$, returns to $C_G(G_1)$ at $(3,5,4)$, and only then closes itself. Call a cycle of $\Gamma_r(G,f)$ \emph{short} if it intersects every vertex set $C_G(G_x)$ at most once.

\begin{theorem}
Let $X$ be a finite set, $G\le S_X$, $Y=\mathrm{Fol}_r(G)$, $F=N_{S_X}(G)$ and consider the action \eqref{Eq:OnTuples} of $F$ on $Y$.

For $f\in F$, let $\bar{f}$ be the permutation of $X/G$ induced by $f$, $\mathrm{Cyc}(\bar f)$ a complete set of cycle representatives of $\bar{f}$, and $c(\bar f,x)$ the cycle of $\bar f$ containing $x\in\mathrm{Cyc}(\bar f)$. Then there is a one-to-one correspondence between the fixed points $\mathrm{Fix}(Y,f)$ of the action of $\langle f\rangle$ on $Y$ and the tuples $(c_x:x\in\mathrm{Cyc}(\bar f))$, where $c_x$ is a short directed cycle of $\Gamma_r(G,f)$ with vertices in $\bigcup_{y\in c(\bar{f},x)}C_G(G_y)$.

Therefore, the number of orbits of $F$ on $Y$ is given by
\begin{displaymath}
    |Y/F| = \frac{1}{|F|}\sum_{f\in F} |\mathrm{Fix}(Y,f)| = \frac{1}{|F|}\sum_{f\in F} \prod_{x\in\mathrm{Cyc}(\bar f)} \gamma_r(G,f,x),
\end{displaymath}
where $\gamma_r(G,f,x)$ is the number of short directed cycles of $\Gamma_r(G,f)$ with vertices in $\bigcup_{y\in c(\bar{f},x)}C_G(G_y)$.
\end{theorem}
\begin{proof}
Let $f\in N_{S_X}(G)$ and let $\Lambda = (\lambda_x:x\in X/G)$ be a rack folder realized as a selection of vertices of the digraph $\Gamma=\Gamma_r(G,f)$, one vertex in each part $C_G(G_x)$. Then $\Lambda f = \Lambda$ if and only if all edges of $\Gamma$ starting in $\Lambda$ also terminate in $\Lambda$, or, in other words, if and only if $\Lambda$ is a disjoint union of short directed cycles. We are done by Burnside's Lemma.
\end{proof}

Note that while $\mathrm{Fol}_r(G)$ is of size $\prod_{x\in X/G}|C_G(G_x)|$ and therefore possibly quite large, the vertex set of every $\Gamma_r(G,f)$ is only of size $\sum_{x\in X/G}|C_G(G_x)|$ so it is relatively easy to construct $\Gamma_r(G,f)$ explicitly. Due to the structure of $\Gamma_r(G,f)$, it is not difficult to count its short directed cycles. Indeed, since the outdegree and indegree of every vertex is equal to $1$, it suffices to trace the cycles starting at vertices of $\bigcup_{x\in\mathrm{Cyc}(\bar f)}C_G(G_x)$ and keep only the short cycles. Of course, we can also determine the number of short cycles from suitable powers of the adjacency matrix of $\Gamma_r(G,f)$.

\section{Open problems}\label{Sc:Open}

We were not able to determine the numbers $r(14)$ and $q(14)$. Note that it suffices to find $r_{\textrm{non-$2$-red}}(14)$ and $q_{\textrm{non-$2$-red}}(14)$ since $r_{\textrm{$2$-red}}(14)$ and $q_{\textrm{$2$-red}}(14)$ are known, cf. Tables \ref{Tb:Racks} and \ref{Tb:Quandles}. The main obstacle in the enumeration is the size of the spaces of rack folders and quandle folders for certain nonabelian permutation groups. For larger orders, it will not be feasible to consider all subgroups of $S_n$ up to conjugacy.

\begin{problem}
Determine the number of isomorphism types of racks and quandles of order $14$.
\end{problem}

It is to be expected that $r(14)$ and $q(14)$ will only slightly exceed $r_{\textrm{$2$-red}}(14)$ and $q_{\textrm{$2$-red}}(14)$, respectively. However, the asymptotic proportion of $2$-reductive racks and $2$-reductive quandles is less predictable:

\begin{problem}\label{Pb:Lims}
What are the limits $\lim_{n\to\infty} \frac{r_{\textrm{$2$-red}}(n)}{r(n)}$ and $\lim_{n\to\infty} \frac{q_{\textrm{$2$-red}}(n)}{q(n)}$, if they exist?
\end{problem}

Let us now look at the structure of $\lmlt{X,*}$ for racks and quandles. Blackburn \cite{Blackburn} proved that for every group $G$ there is a quandle $(X,*)$ on some set $X$ such that $\lmlt{X,*}$ is isomorphic to $G$. But not every permutation group is rack/quandle admissible. In view of the results in Section \ref{Sc:Representation}, a subgroup $G\le S_X$ is rack admissible (resp. quandle admissible) if and only if there are $(\lambda_x\in C_G(G_x):x\in X/G)$ (resp. $(\lambda_x\in Z(G_x):x\in X/G)$) such that $\langle \bigcup_{x\in X/G} \lambda_x^G\rangle = G$.

\begin{problem}\label{Pb:LMlt}
Describe a large or algebraically significant class of subgroups $G$ of $S_X$ that are not rack/quandle admissible, that is, for which there is no rack/quandle $(X,*)$ such that $\lmlt{X,*}=G$.
\end{problem}

To shed some light on Problems \ref{Pb:Lims} and \ref{Pb:LMlt}, we offer the following observations about left multiplication groups of quandles of order $10$. These and similar facts can be verified using the library of racks and quandles available on the web page of the first author.

There are $102771$ quandles of order $10$ and their left multiplication groups form a set of $471$ non-equivalent permutation groups. The number of quandles associated with a given permutation group varies greatly. There are $63$ permutation groups with a unique quandle, $84$ with $2$ quandles and $22$ with $3$ quandles. On the other side of the spectrum, there are five permutation groups that account for $20084$, $17336$, $12033$, $6359$ and $6284$ quandles, respectively. Up to isomorphism, these groups are $C_3\times C_2^3$, $C_2^4$, $C_3\times C_2^2$, $C_3^2\times C_2$ and $C_4\times C_2^2$, respectively. The most prolific group has generators $(2,5,3)(7,9)(8,10)$, $(8,10)$ and $(4,6)(7,9)$.

There is a unique non-solvable group among the $471$ permutation groups, namely a copy of $S_5$ generated by $(1,4,6,3,2)(5,8,7,10,9)$ and $(3,8)(5,9)(6,10)$. There are $2$ quandles associated with this group. Furthermore, there are $247$ non-nilpotent groups with $3383$ associated quandles, $320$ non-abelian groups with $4239$ associated quandles, and $59$ elementary abelian $2$-groups with $35091$ associated quandles.

\begin{problem}
Let $G\le S_X$. Show how Burnside's Lemma can be effectively used to count the orbits of the action of $N_{S_X}(G)$ on $\mathrm{Env}_r(G)$ or $\mathrm{Env}_q(G)$.
\end{problem}

Finally, our computational data support the following conjecture:

\begin{conjecture}
Let $p$ be a prime. Then $r_{\emph{med}}(p) - r_{\emph{$2$-red}}(p) = p-2$ and $q_{\emph{med}}(p) - q_{\emph{$2$-red}}(p) =  p-2$. In particular, every medial rack of order $p$ that is not $2$-reductive is a quandle.
\end{conjecture}

\section*{Acknowledgement}

We thank Alexander Hulpke for the reference \cite{Holt}, P\v{r}emysl Jedli\v{c}ka for the numbers of $2$-reductive racks of small orders, Victoria Lebed for the references \cite{AshfordRiordan} and \cite{Blackburn}, David Stanovsk\'y for useful discussions about Burnside's Lemma in the context of affine meshes, Glen Whitney for a clarification about $2$-reductivity in racks, and anonymous referees for useful comments. The calculations took place at the High Performance Computing Cluster of the University of Denver.

\end{document}